\documentclass[12pt]{amsart}
\usepackage{amsmath, amssymb}
\newfont {\cyr} {wncyr10}
\renewcommand{\labelenumi}{{(\roman{enumi})}}

\usepackage{amsfonts}
\usepackage{amsmath}
\usepackage{amssymb}
\usepackage{setspace}
\usepackage{multicol}

\newtheorem{theorem}{Theorem}[section]
\newtheorem{lemma}[theorem]{Lemma}
\newtheorem{proposition}[theorem]{Proposition}

\newtheorem{hypothesis}[theorem]{Hypothesis}

\newcounter{claim}[theorem]

\newcounter{cclaim}[theorem]

\def \udot {{}^{\textstyle .}}

\newcommand{\F}{\mathrm{F}}\newcommand{\B}{\mathrm{B}}\newcommand{\M}{\mathcal{M}}

\newcommand{\Aut}{\mathrm{Aut}}

\newcommand{\Syl}{\mathrm{Syl}}\newcommand{\syl}{\mathrm{Syl}}

\newcommand{\GF}{\mathrm{GF}}

\newcommand{\SL}{\mathrm{SL}}

\newcommand{\PGL}{\mathrm{PGL}}
\newcommand{\PSL}{\mathrm{PSL}}\newcommand{\PSp}{\mathrm{PSp}}
\newcommand{\Sym}{\mathrm{Sym}}
\newcommand{\Alt}{\mathrm{Alt}}
\newcommand{\Dih}{\mathrm{Dih}}

\def \syl {\hbox {\rm Syl}}\def \Syl {\hbox {\rm Syl}}

\def \ov {\overline}

\def \Aut{ \mathrm {Aut}}

\def \Fi {\mbox {\rm Fi}}

\def \J{\mbox {\rm J}}

\def \B{\mbox {\rm B}}

\def \M{\mbox {\rm M}}

\def \McL{\mbox {\rm McL}}

\def \PSU {\mbox {\rm PSU}}

\begin{document}
\renewcommand{\labelenumi}{(\roman{enumi})}

\title  {On strongly $p$-embedded subgroups of Lie rank $2$}
 \author{Chris Parker}
  \author{Gernot Stroth}

\address{Chris Parker\\
School of Mathematics\\
University of Birmingham\\
Edgbaston\\
Birmingham B15 2TT\\
United Kingdom} \email{c.w.parker@bham.ac.uk}

\address{Gernot Stroth\\
Institut f\"ur Mathematik\\ Universit\"at Halle - Wittenberg\\
Theordor Lieser Str. 5\\ 06099 Halle\\ Germany}
\email{gernot.stroth@mathematik.uni-halle.de}

\email {}

\def \eps {\epsilon}
\def\diag {\mathrm {diag}}
\def\SO {\mathrm{SO}}\def\GO {\mathrm{GO}}
\date{\today}

\maketitle \pagestyle{myheadings}

\markright{{\sc }} \markleft{{\sc Chris Parker and Gernot Stroth}}

\section{Introduction}

For a prime $p$, a proper subgroup $H$ of the finite group $G$ is  \emph{strongly $p$-embedded in
$G$} if $p$ divides $|H|$  and $p$ does not divide $|H \cap H^g|$ for all $g \in G\setminus H$. A
characteristic  property  of strongly $p$-embedded subgroups is that  $N_G(X)\le H$ for any
non-trivial $p$-subgroup $X$ of $H$. Strongly $p$-embedded subgroups appear in the final stages of
one of the programmes to better understand the classification of the finite simple groups \cite
{MSS}. The almost simple groups  with strongly $2$-embedded subgroups were determined by Bender
\cite{Bender} and Suzuki \cite{Suzuki}.

Recall that a \emph{$\mathcal K$-group} is a group in which every composition factor is from the
list of ``known" simple groups. That is, every simple section is either a cyclic group of prime
order, an alternating group, a group of Lie type or one of the twenty six sporadic simple groups.
In the classification of groups of local characteristic $p$, certain normalizers of non-trivial
subgroups are assumed to be $\mathcal K$-groups.

The objective of this paper is to extend the results of \cite[Corollary 1.4]{ParkerStroth} to cover
some of the Lie type groups of Lie rank 2 defined in characteristic $p$. Further results related to
local characteristic $p$ identifications of rank $2$ Lie type groups can be found in the work of
Parker and Rowley \cite{WBN}. Our main result in this article is the following theorem.

\begin{theorem}\label{thma}
Suppose that $p$ is an odd prime, $G$ is a finite group and that $N_G(T)$ is a $\mathcal K$-group
for all non-trivial $2$-subgroups $T$ of $G$. If $H$ is a  strongly $p$-embedded subgroup of $G$,
then $F^*(H)$ is not a simple Lie type group of Lie rank $2$ defined in characteristic $p$ unless
perhaps $F^*(H) \cong \PSL_3(p)$.
\end{theorem}

Suppose that $G$ and $H$ are as in Theorem~\ref{thma}. Then the main theorems in
\cite{ParkerStroth} can be applied when the $p$-rank of $C_H(t)$, $m_p(C_H(t))$, is at least $2$
for every involution $t$ in $H$. For groups $H$ with $F^*(H)$ a simple Lie type group defined in
characteristic $p$ and Lie rank $2$, we use \cite[Proposition 9.1]{ParkerStroth} to see that if
$m_p(C_H(t))\le 1$ for some involution in $H$, then $F^*(H) \cong \PSp_4(p)$ or $\PSL_3(p)$. Thus,
because of \cite[Theorem 1.2]{ParkerStroth}, to prove Theorem~\ref{thma}, we  now have to prove the
following theorem.

\begin{theorem}\label{thmone} Suppose that $p$ is an odd prime, $G$ is a finite group, $N_G(T)$ is a $\mathcal K$-group
for all non-trivial $2$-subgroups $T$ of $G$ and $H \le G$. If $H$ is strongly $p$-embedded in $G$,
then $F^*(H) \not \cong \PSp_4(p)$.
\end{theorem}

The methods that we apply in this paper cannot be easily extended  to eliminate the possibility
that $F^*(H)\cong \PSL_3(p)$. This is because in $\PSp_4(p)$ there are involutions which have
centralizer of $p$ rank $2$ whereas in $\PSL_3(p)$ there are no such involutions. This means that
it is much more difficult to control the structure of the centralizer of an involution of $G$ in
the latter case. Indeed so far we have been unable to eliminate this configuration, though the
classification of the finite simple groups shows that it cannot occur.

In Section 2 we present the facts about $\PSp_4(p) \cong \Omega_5(p)$ and background results  from
\cite{ParkerStroth} about strongly $p$-embedded subgroups. We prove  Theorem~\ref{thmone} in
Section 3. Our notation is standard, as can be found in any of \cite{Aschbacher, Gorenstein,
huppert}. It is also consistent with that in \cite{ParkerStroth}.

\section{Preliminaries}

In this section we present facts about $\PSp_4(p) \cong \Omega_5(p)$ and present background lemmas
which we require for the proof of Theorem~\ref{thmone}.

\begin{lemma}\label{O5facts} Suppose that $(V,Q)$ is a $5$-dimensional non-degenerate orthogonal space over  $\GF(p)$ where $p$ is  a prime. Let  $X =  \SO(V,Q) \cong \SO_5(p)$ and $Y \cong \Omega(V,Q)
 \cong \Omega_5(p)$.
Assume $\beta =\{e_1,e_2,e_3,e_4,e_5\}$ is an orthonormal basis for $(V,Q)$ and set $\eps =
(-1)^{(p-1)/2}$.
\begin{enumerate}
\item $\Aut(Y) \cong X$.
\item $Y$ has exactly two conjugacy classes of involutions, one has representative
$b=\diag(-1,-1,1,1,1)$ with respect to $\beta$ and the other has representative $a= \diag
(1,-1,-1,-1,-1)$ with respect to $\beta$.
\item As quadratic spaces, $[V,b]$ has $\eps$-type  and $[V,a]$ has $+$-type.
\item $C_Y(b)$ has a normal subgroup $L$ isomorphic to $\Omega_3(p)$ such that $C_Y(b)/L \cong
\mathrm O_2^\eps(p) \cong \Dih (2(p-\eps))$.
\item $a$ is $2$-central, $O^{2}(C_X(a)) \cong \Omega_4^+(p)$, $C_X(a)/O^{2}(C_X(a))$ is a fours
group and $O_{p'}(C_X(t))$ is a $2$-group.
\item $C_Y(b)$ contains exactly four conjugacy classes of fours groups containing $b$. They have representatives $F_1 =
\langle b, \diag (1,-1,-1,1,1) \rangle$ which contains three $Y$-conjugates of $b$, $F_2 = \langle
b, a \rangle$ which contains  $b$ and two $Y$-conjugates of $a$, $F_3 = \langle b, \diag
(1,1,-1,-1,1) \rangle$ which contains two $Y$-conjugates of $b$ and a $Y$-conjugate of $a$
and, finally, \begin{eqnarray*} F_4 &=& \langle b, r_vr_w\mid v \in [V,a], w \in C_W(a), Q(v)=Q(w)\\
&&\;\;\;\;\;\;\;\;\;\;\;\;\;\mbox{ a non-square and } r_x \mbox{ a reflection for }x\in \{v,w\}
\rangle\end{eqnarray*} which contains three $Y$-conjugates of $b$.

\end{enumerate}
\end{lemma}

\begin{proof} Parts (i) and (ii) are well known and can be read from \cite[Theorem 2.5.12, Table
4.5.1]{GLS3}. We have that $[V,b]=\langle e_1,e_2\rangle$ and  $[V,a]=\langle
e_2,e_3,e_4,e_5\rangle$. The type of these spaces is $\eps = (-1)^{m(p-1)/2}$ where $m$ is the
dimension of the space. This gives (iii). In $Z= \mathrm O_5(p)$, $C_Z(b)$  stabilizes  $[V,b]$ and
is thus isomorphic to $\mathrm O_2^\eps(p) \times \mathrm O_3(p)$. Note that $X$ has index $2$ in
$Z$ and $Y$ has index $4$ in $Z$ with $Z/Y$ elementary abelian. We obtain $C_Y(b)$ by taking the
subgroup of $C_Z(b)$ which contains $\Omega_2^\eps(p) \times \Omega_3(p)$ together with the
elements of $C_Z(b)$ which have determinant and spinor norm both $-1$ when projected to both
$\mathrm O ^\eps_2(p)$ and to $\mathrm O_3(p)$. So (iv) holds. Part (v) follows in the same way as
(iv).

It is straight forward to calculate all the $C_Y(b)$-classes fours groups as indicated  in (vi) by
considering the action of the fours groups on $[V,b]$.
\end{proof}

We will need the following consequence of the Thompson Transfer Lemma \cite[15.16]{GLS2}.

\begin{lemma}\label{TTLG} Suppose that $G$ is a group, $L= O^2(L) \le G$, $T \in \Syl_2(L)$ and $S \in \Syl_2(G)$
with $S\ge T$. Assume that there exists a dihedral subgroup $D \le S$ such that $D\cap T=1 $, $TD =
S$ and $C \le D$ is cyclic of index $2$ in $D$. If the involutions of $D\setminus C$ are not
$G$-conjugate to any involution in $TC$ and the involution in $C$ is not $G$-conjugate to any
involution in $T$, then $G$ has a normal subgroup $K$ such that $\syl_2(L) \subseteq \Syl_2(K) $
and $G= KD$.
\end{lemma}

\begin{proof} Since $|S:TC|=2$, the Thompson Transfer Lemma implies that $G$ has a normal subgroup $G_0$ of
index $2$. Then $L= O^2(L) \le O^2(G) \le G_0$ and  $S_0 = S \cap G_0= TD \cap G_0 = T(D\cap G_0)$.
Let $D_0 = D\cap G_0$. Then either $D_0= C$, $D_0 $ is dihedral of order $|D|/2$ or $|D_0|=2$.
Suppose that $D_0= C$. Then, because $L= O^2(L) \le O^2(G)$ and because the involution of $C$ is
not $G$-conjugate to an element of $T$, the Thompson Transfer Lemma applies repeatedly to give a
normal subgroup $K$ of $G_0$ which contains $L$  and has  $\syl_2(L) \subseteq \Syl_2(K) $ and $G=
KD$. Thus the lemma holds in this case. If $D_0$ is dihedral, then we may apply induction to $G_0$
and obtain  the required result. Thus we may assume that $D_0$ has order $2$. But then again the
Thompson Transfer Lemma gives the result.
\end{proof}

We will call on various  results from \cite[Section 3]{ParkerStroth} and for the sake of clarity we
repeat their statements here.

\begin{lemma}\label{main1} Suppose that $G$ is a group, $p$ is a prime and $H$ is a strongly
$p$-embedded subgroup of $G$. Then the following statements hold.
\begin{enumerate}
\item If $K \le G$, $K \not \le H$ and $H \cap K$ has order divisible by $p$, then $H \cap
K$ is a strongly $p$-embedded subgroup of $K$;
\item $\syl_p(H \cap K) \subseteq \syl_p(K)$; and
\item if $m_p(H) \ge 2$, then $O_{p'}(G) = \bigcap H^G$.
\end{enumerate}
\end{lemma}

\begin{proof}See \cite[Lemma 3.2]{ParkerStroth}.
\end{proof}

\begin{proposition} \label{SE-p2} Suppose that $p$ is a prime, $X$ is a $\mathcal K$-group and
$K=F^*(X)$ is simple. Let  $P \in \syl_p(X)$ and $Q= P\cap K$. If $m_p(P) \ge 2$ and $X$ possesses
a strongly $p$-embedded subgroup, then $K$ has a strongly $p$-embedded subgroup  and  $p$ and $K$
are as follows.
\begin{enumerate}
\item $p$ is arbitrary, $a \ge 1$ and $K \cong \PSL_2(p^{a+1})$,
or $\PSU_3(p^a)$,  ${}^2\B_2(2^{2a+1})$ $(p=2)$ or ${}^2{\rm G}_2(3^{2a+1})$ $(p=3)$ and $X/K$ is a
$p'$-group.
\item  $p > 3$, $K \cong \Alt(2p)$ and
$|X/K| \le 2$.
 \item $(K,p)$ is one of the pairs:  $(\PSL_2(8),3)$, $(\PSL_3(4),3)$,
 $(\M_{11},3)$, $({}^2\B_2(32),5)$, $({}^2\F_4(2)^\prime,5)$, $(\McL,5)$, $(\Fi_{22},5)$,
 $(\J_4,11)$.
\end{enumerate}\end{proposition}

\begin{proof} See \cite [7.6.1]{GLS3}. \end{proof}

We let $\mathcal E$ be the set of pairs $(K,p)$ in  Proposition~\ref{SE-p2} (i), (ii) and (iii)
with $p$ odd.

\begin{lemma}\label{main2} Suppose that $G$ is a group, $p$ is a prime and $H$ is a strongly
$p$-embedded subgroup of $G$. Set $\ov G = G/O_{p'}(G)$ and assume  further that $\ov H \neq \ov
G$. Then
\begin{enumerate}
\item $\ov H$ is strongly $p$-embedded in $\ov G$;
\item $F^*(\ov G)$ is a non-abelian simple group; and
\item if $G$ is a $\mathcal K$-group and $m_p(G) \ge 2$, then $(F^*(\ov G),p) \in \mathcal E$.
\end{enumerate}
\end{lemma}

\begin{proof}See \cite[Lemma 3.3]{ParkerStroth}.
\end{proof}

The next Lemma says that groups with strongly $p$-embedded subgroups have good control of
$G$-fusion in certain circumstances.

\begin{lemma}\label{control1} Suppose that $p$ is an odd prime, $G$ is a  group and  $H$ is strongly $p$-embedded of  $G$. Assume that for all involutions $t \in H$,  $p$ divides
$|C_H(t)|$. Then for all involutions $t \in H$, $t^G \cap H = t^H$.
\end{lemma}

\begin{proof}See \cite[Lemma 3.5]{ParkerStroth}.\end{proof}

\section{The proof of Theorem~\ref{thmone}}

In this section we prove Theorem~\ref{thmone}. Assuming that it is false we, work under the
following hypothesis.

\begin{hypothesis} $H$ is strongly $p$-embedded in a finite group $G$ and
$H^* = F^*(H)= E(H) \cong \PSp_4(p)\cong \Omega_5(p)$ for an odd prime $p$. Furthermore, $N_G(T)$
is a $\mathcal K$-group for all non-trivial $2$-subgroups $T$ of $G$.
\end{hypothesis}

\begin{lemma}\label{E=} Either $H = H^*$ or $H \cong \SO_5(p)$.
\end{lemma}

\begin{proof} This follows from Lemma~\ref{O5facts} (i) as $F^*(H) \cong\Omega_5(p)$.
\end{proof}

Let $s$ and $t$ be representatives of the two conjugacy classes of involutions in $H^*$ and assume
that $t$ is $2$-central in $H$. Furthermore assume that $F_3 = \langle s,t \rangle$ is the centre
of a Sylow $2$-subgroup $T$  of $C_H(s)$. Notice that $F_3$ contains exactly one $H$-conjugate of
$t$ by Lemma~\ref{O5facts} (iv). By Lemmas~\ref{main1} (iii) and \ref{E=}, $O_{p'}(G) =1$. Thus, as
$H^*$ is simple and contains a Sylow $p$-subgroup of $G$, if $K$ is a normal subgroup of $G$, $K
\cap H \ge H^*$. Thus, because of Lemma~\ref{main1} (iii), we may suppose by induction that $G$ is
a simple group. Since we may also assume that $G$ is not a $\mathcal K$-group, $G$ does not possess
a strongly $2$-embedded subgroup. In particular, by \cite[17.13]{GLS2}, we have the following
lemma.

\begin{lemma}\label{not2embedded} There is an involution $r \in H$ such that $C_G(r) \not \le
H$.\qed
\end{lemma}

The next lemma is the key result of the article. It shows that we can be much more specific about
the involution class which has centralizer not contained in $H$.

\begin{lemma}\label{Ctnotin} $C_G(t) \not \le H$.
\end{lemma}

\begin{proof} Suppose that $C_G(t) \le H$ and let $S \in \Syl_2(C_H(t))$ with $S \ge T$. Then, as $t$ is
$2$-central in $H$, we have $S \in \Syl_2(H)$. Furthermore, $t$ is the unique involution in $Z(S)$
and so $N_G(S) \le H$. Thus $S \in \Syl_2(G)$. Since $H$ controls $G$-fusion of involutions in $H$
by Lemma~\ref{control1}, $S \in \Syl_2(G)$, $G= O^2(G)$ and there are involutions in
$\SO_5(p)\setminus \Omega_5(p)$, the Thompson Transfer Lemma implies that $H = H^*$. Hence
Lemma~\ref{not2embedded} implies $C_G(s) \not \le H$. Since $s$ and $t$ are not $G$-conjugate by
Lemma~\ref{control1}, $t$ is the unique $G$-conjugate of $t$ in $F_3=Z(T)$. It follows that $N_G(T)
\le C_G(t) \le H$. Hence  $T$ is a Sylow $2$-subgroup of $C_G(s)$. From Lemma~\ref{O5facts}(iv),
$C_H(s)$ has cyclic Sylow $p$-subgroups of order $p$. As $H$ is strongly $p$-embedded in $G$,
Lemma~\ref{main1} (ii) implies that $C_G(s)$ has Sylow $p$-subgroups of order $p$. Let $P \in
\syl_p(C_H(s))$ and set $L= \langle P^{C_H(s)}\rangle$ and $R=O_{p'}(C_G(s))$. Then
Lemma~\ref{O5facts} (iv) implies that $L \cong \Omega_3(p)\cong\PSL_2(p)$. Assume that $R \not \le
H$.  Let $K$ be a component of $C_G(s)$ such that $p$ divides $|K|$. Then $K \not \le R$ and so $K$
centralizes $R$. Hence $R \le C_G(P) \le H$ which is a contradiction. Therefore $E(C_G(s)) \le R$.
As $P$ is not normal in $L$, $O_p(C_G(s))=1$. Hence $F(C_G(s)) \le R$ and, so $F^*(C_G(s)) \le R$.
Set $U = O(R)$ and let $E$ be a fours group of $L$. Then, for $e \in E^\#$, we have $\langle
e,s\rangle$ is a fours subgroup of $C_H(s)$ which is conjugate to $F_3$ by Lemma~\ref{O5facts}
(vi). It follows that $es$ is $H$-conjugate to $t$. Hence $C_{U}(e) = C_{U}(es)\le C_G(es) \le H$,
by hypothesis. Since $U= \langle C_{U}(e)\mid e \in E^\#\rangle$, we conclude that $U  \le H$ and
that $[U,L]=1$. Note that $T \cap R \in \syl_2(R)$ and $T \cap R \le O_{2}(C_H(s))$. Assume that $p
\neq 3$. Then $O_2(C_H(s))$ commutes with $L$. Hence $O_2(C_G(s)) $ is also centralized by $L$.
Thus $L$ commutes with $UO_2(R)$ and consequently $L$ commutes with $F(C_G(s))$. It follows that
$E(C_G(s)) \neq 1$. Let $K$ be a component of $C_G(s)$. Then $T \cap K \in \syl_2(K)$ and is
centralized by $L$. It follows that $L$ normalizes $K$. Since $L$ is a non-abelian simple group,
the Schreier property of simple groups implies that $L$ induces inner automorphisms of $K$
(remember $C_G(s)$ is a $\mathcal K$-group). As $L$ has order divisible by  $p$ and $K$ does not,
$[K,L]=1$. But then $K \le C_G(P) \le H$ and we have a contradiction. Hence, if $p \neq 3$, then $R
\le H$. Suppose that $p=3$.  If $K$ is a component of $C_G(s)$, then, as $K \le R$, $K$ is a
$3'$-group. Hence, as $C_G(s)$ is a $\mathcal K$-group,
 $K/Z(K)$ is  a Suzuki group. This, however, contradicts $2^6\le |K|_2 \le
|C_G(s)|_2=|T|=2^5$. Hence $E(C_G(s)) = 1$. Since $U \le H$, $C_{H}(s)$ is a $\{2,3\}$-group and $R
\cap H$ is
 a $2$-group, we infer that $U=1$ and $F=F^*(C_G(s)) =O_2(C_G(s))\le T\le C_H(s)$. Now $O_2(C_H(s))$ is elementary
abelian of order $16$ and so we have $F= O_2(C_H(s))$. But  $F$ contains exactly $5$ conjugates of
$t$ and so, as $H$ controls $G$-fusion of involutions in $H$ and, from the subgroup structure  of
$\PSp_4(3)\cong \PSU_4(2)$, $N_H(F)/F \cong \Alt(5)$, we infer that $C_G(s) \le H$, which is a
contradiction. Thus $R \le H$.

As $C_G(s) > RC_H(s)$, it follows from Lemma~\ref{main2}(ii) that $C_G(s)/R$ is an almost simple
group and $C_H(s)/R$ is strongly $p$-embedded in $C_G(s)/R$.  As $H=H^*\cong \Omega_5(p)$, the
structure of $C_H(s)$ given in Lemma~\ref{O5facts} (iv) implies $C_H(s)/L \cong \mathrm
O_2^\eps(p)$ where $\eps = (-1)^{(p-1)}$ and the extension is split. In particular, $C_H(s)/L$ has
dihedral Sylow $2$-subgroups. Set $W = T\cap L$ and let $D \le T$ be a dihedral $2$-group with $D
\cap W=1$ and $T= DW$. Let $C \le D$ be cyclic of index $2$ in $D$. Furthermore, make these choices
so that $s \in C$. Then $T_0 = CW$ has index $2$ in $T$ and, by Lemmas~\ref{O5facts} and
\ref{control1} all the fours groups which contain $s$ and are contained in $T_0$ are
$C_G(s)$-conjugate to $F_3$ and the fours groups containing $s$ and not contained in $T_0$ are not
$C_G(s)$-conjugate to $F_3$. Thus, as the involution in $C$, namely $s$, is not $C_G(s)$-conjugate
to an element of $W$ and $L= O^2(L)$, Lemma~\ref{TTLG} implies that $C_G(s)$ contains a normal
subgroup $K$, with $\Syl_2(L) \subseteq \Syl_2(K)$ and $C_G(s) = DK$. In particular, $K$ has
dihedral Sylow $2$-subgroups.

Since $L$ is simple when $p \ge 5$ and when $p=3$, $L$ has abelian Sylow $2$-subgroups of order
$4$, it follows from \cite[Theorem 16.3 pg.462]{Gorenstein} that $K/O(K) \cong \PSL_2(r^a)$ for
some odd prime $r$ or $K/O(K)\cong \Alt(7)$. Suppose first that $K/O(K) \cong \Alt(7)$. Then as $L$
contains a Sylow $2$-subgroup of $K$, we infer that $L \cong \PSL_2(7)$ and that $p=7$. Since
$TLR/R \cong \SO_3(7) \cong \PGL_2(7)$ and $\PGL_2(7)$ has no subgroup of index $7$, we have a
contradiction. So suppose that $K/O(K) \cong \PSL_2(r^a)$. Then, as $\PSL_2(p) \cong L \le H\cap K$
and $H\cap K$ is strongly $p$-embedded in $K$, we deduce from Dickson's List of subgroups of
$\PSL_2(r^a)$ \cite[8.27, pg. 213]{huppert} that $r\neq p$ and that $p=5$ or $p=3$. In  both cases
$L$ has elementary abelian Sylow $2$-subgroups of order $4$ and, from the structure of $\PSp_4(3)$
and $\PSp_4(5)$,  $O_{p'}(C_{C_H(s)}(P))$ is also a four group, where we recall that  $P \in
\syl_p(L)$. It follows that $O(K)=1$ and that $P$ is self-centralizing in $K$. On the other hand,
we know that the centralizer of $P$ in $K$ has order $(r^a+1)/2$ or $(r^a-1)/2$ and $K$ has abelian
Sylow $2$-subgroups. Therefore we infer that either $p=5$ and $r^a =11$ or $p=3$ and $r^a=5$. If
$p=5$, then $DL/C_D(L) \cong \PGL_2(5)$. But $\PGL_2(11)$ contains no such subgroup and this is a
contradiction. Thus $p=3$. Hence $H \cong \PSp_4(3)$, $L \cong \PSL_2(3)\cong \Alt(4)$, $D \cong
\Dih(8)$ and $K \cong \PSL_2(5)\cong \Alt(5)$. We also have $DL/C_{DL}(L) \cong \Sym(4)$ and so
$DK/C_D(K)\cong \Sym(5)$ and $|C_{D}(K)| =2^2$ . Since $P$ centralizes $C_{D}(K) $ and
$C_{DK/C_D(K)}(PC_D(K)/C_D(K)) \cong \Sym(3) \times \Sym(2)$, we have $|C_{DK}(P)| =2^33$. This
shows that  $C_{DK}(P) \not \le H$, which is a contradiction. Therefore we have shown $C_G(t) \not
\le H$ as claimed.
\end{proof}

We may now conclude the proof of Theorem~\ref{thmone}.

\begin{proof}[Proof of Theorem~\ref{thmone}] Since $C_G(t) \not\le H$ by Lemma~\ref{Ctnotin} and $m_p(C_G(t)) =2$,
Lemmas~\ref{main1} (iii) and \ref{main2} imply that $O_{p'}(C_G(t))\le H$ and, as $C_G(t)$ is a
$\mathcal K$-group, $(F^*(C_G(t)/O_{p'}(C_G(t))),p)\in \mathcal E$.  In particular, we have
$O_{p'}(C_G(t)) \le O_{p'}(C_H(t)) = O_2(C_H(t))$ by Lemma~\ref{O5facts} (v).

Suppose that $p=3$. Then, as $C_H(t)$ acts irreducibly on $O_2(C_H(t))/\langle t \rangle$, either
$O_{p'}(C_G(t))=\langle t \rangle$ or  $O_{p'}(C_G(t))= O_2(C_H(t))$. If $O_{p'}(C_G(t))=\langle
t\rangle$, then  $|O_2(C_H(t)/\langle t\rangle)| = 16$. As $C_H(t)/O_{p'}(C_G(t))$ is strongly
$p$-embedded in $C_G(t)/O_{p'}(C_G(t))$, this contradicts the structure of the groups in
$\mathcal{E}$ given in Proposition~\ref{SE-p2}  (see also \cite[Proposition 2.7]{ParkerStroth}).
Thus $O_2(C_H(t)) = O_2(C_G(t))= F^\ast(C_G(t))$ is extraspecial of order $2^5$ and plus-type.
Therefore we have $C_G(t)/O_2(C_G(t))$ is isomorphic to a subgroup of $\mathrm O^+_4(2)$. This then
forces $C_G(t) = C_H(t)$, which is impossible. Hence $p\neq 3$.

So assume that $p \ge 5$. Then, by Lemma~\ref{O5facts} (iv), $O_{p'}(C_G(t)) =\langle t\rangle$ and
$C_H(t)/\langle t\rangle$ is not soluble. The  structure of $F^*(C_H(t))/\langle t \rangle$
together with Proposition~\ref{SE-p2} implies that $p=5$ and that $F^*(C_G(t))/O_{p'}(G_G(t)) \cong
\Alt(10)$. Since $F^*(C_H(t)) \ge \langle t \rangle$, it follows that $F^*(C_G(t)) \cong 2\udot
\Alt(10)$. Now from the structure of $2\udot\Alt(10)$, we deduce that $X=F^*(C_H(t)) \cong
\SL_2(5)*\SL_2(5)$ and $N_{F^*(C_G(t))}(X)/X$ is cyclic of order $4$. This is unfortunately
incompatible with the structure of $N_G(X)$ in $H$ where we have either $N_H(X)/X$ is cyclic of
order $2$ or is a fours group. This contradiction shows that $H$ cannot be strongly $p$-embedded in
$G$ and completes the proof of Theorem~\ref{thmone}.
\end{proof}

As mentioned in the introduction, Theorem~\ref{thma} follows by combining \cite[Proposition 9.1,
Theorem 1.2]{ParkerStroth} with Theorem~\ref{thmone}.

\end{document}